\documentclass[10pt]{article}

\usepackage{amsmath,amsthm}
\usepackage{amssymb}
\usepackage{breqn}
\usepackage{enumerate}
\usepackage{graphicx}
\usepackage{bbm}
\usepackage{color}

\providecommand{\diam}{\mathop{\rm diam}\nolimits}

\newtheorem{theorem}{Theorem}[section]
\newtheorem{corollary}[theorem]{Corollary}
\newtheorem{lemma}[theorem]{Lemma}
\newtheorem{proposition}[theorem]{Proposition}
\newtheorem{question}[theorem]{Question}

\newtheorem{conjecture}[theorem]{Conjecture}

\theoremstyle{definition}

\title{Contractive Families on Compact Spaces}
\author{Luka Mili\'{c}evi\'{c}}
\begin{document}
\date{}\maketitle
\footnotesize 
\begin{center}\textcolor{blue}{This is a preprint and not an identical copy of a paper `Contractive Families on Compact Spaces' of the same author, to appear in \emph{Mathematika}}\end{center}
\normalsize
\vspace{6pt}
\begin{abstract} A family $f_1,...,f_n$ of operators on a complete metric space $X$ is called contractive if there exists $\lambda < 1$ such that for any $x,y$ in $X$ we have $d(f_i(x),f_i(y)) \leq \lambda d(x,y)$ for some $i$. Stein conjectured that for any contractive family there is some composition of the operators $f_i$ that has a fixed point. Austin gave a counterexample to this, and asked if Stein's conjecture is true if we restrict to compact spaces. Our aim in this paper is to show that, even for compact spaces, Stein's conjecture is false. \end{abstract}
\let\thefootnote\relax\footnote{2000 Mathematics Subject Classification 47H09 (primary), 54E45(secondary)}
\section{Introduction}
Let $(X,d)$ be a (non-empty) complete metric space. Given $n$ continuous functions $f_1, f_2, \dots, f_n : X \rightarrow X$, and a real number $\lambda \in (0,1)$, we call $\{f_1, f_2, \dots, f_n\}$ a \emph{$\lambda$-contractive} family if for every pair of points $x, y$ in $X$ there is $i$ such that $d(f_i(x), f_i(y)) \leq \lambda d(x,y)$. Further, we say that  $\{f_1, f_2, \dots, f_n\}$ is a \emph{contractive} family if it is a $\lambda$-contractive family for some $\lambda \in (0,1)$. In particular, when $f$ is a function on $X$ and $\{f\}$ is a contractive family we say that $f$ is a \emph{contraction}. Recall the well-known theorem of Banach~\cite{Banach} which says that any contraction on a complete metric space has a unique fixed point.\\
In~\cite{Stein}, Stein conjectured the following generalisation of the theorem of Banach:\\
\indent Let $\{f_1, f_2, \dots, f_n\}$ be a $\lambda$-contractive family on a complete metric space. Then some composition of $f_1, f_2, \dots, f_n$ (i.e. some word in $f_1, \dots,f_n$) has a fixed point.\\
In ~\cite{Austin}, Austin constructed a counterexample to this statement.
\begin{theorem}[\cite{Austin}] There is a complete metric space $(X,d)$ with a contractive family of operators $\{f, g\}$, such that no word in $f, g$ has a fixed point.
\end{theorem} 
Furthermore, Austin asked if this is possible in a compact space.
\begin{question}[\cite{Austin}] Does every contractive family of operators on a compact space have a composition with a fixed point?
\end{question}
The purpose of this paper is to show that even with the additional assumption of compactness, there still need not be a fixed point.\\ 
\begin{theorem}\label{thm1} There is a compact metric space $(X,d)$ with a contractive family of operators $\{f, g\}$, such that no word in $f, g$ has a fixed point.\end{theorem}
We remark that our construction provides the counterexample for any given $\lambda \in (0,1)$.

\section{Outline of the construction of the counterexample}
Since our construction of the counterexample which proves the Theorem~\ref{thm1} is fairly complicated, we provide some motivation for our steps. Suppose that we are in the setting of Theorem~\ref{thm1}. Given any set $S$ of points in our space, since $X$ is compact, we know that $S$ is bounded. Let $D$ be its diameter. Then, it is not hard to show that the diameter of one of $f(S)$ or $g(S)$ is at most $4\lambda D$. To see this, pick any point $x$ in $S$ and consider set $S_f = \{y \in S: \lambda d(x,y) \geq d(f(x), f(y))\}$ and let $S_g = S \setminus S_f$. If $S_f = S$, we are done, so suppose that $S_g \not=\emptyset$. If we can find $y\in S_f$ such that $\lambda d(y, z) \geq d(f(y), f(z))$ for all $z \in S_g$, then, by looking at distance form $f(y)$, the diameter of $f(S)$ does not exceed $4\lambda D$. On the other hand, if there is no such $y$, then each point of $g(S_f)$ is on distance at most $\lambda D$ from some point in $g(S_g)$ which are on distance at most $\lambda D$ from $g(x)$, so $g(S)$ has diameter at most $4\lambda D$.\\
This simple observation leads us to the idea that instead of considering distances between each pair of points in the wanted space, thinking about diameters of sets should be much more convenient in our problem. With this in mind, we develop the notion of `diameter spaces', which will play a key role throughout our construction. Due to their importance, we include the proper definition.\\

Let $X$ be a non-empty set. Given a collection $\mathcal{D}$ of subsets of $X$, we call $(X, \mathcal{D})$ a \emph{diametrisable space} provided the following conditions are met,
\begin{enumerate}
\item Given $x, y \in X$ there is $U \in \mathcal{D}$ with $x, y \in U$
\item If $U, V \in \mathcal{D}$ and $U \cap V \not=\emptyset$ then $U \cup V \in \mathcal{D}$
\end{enumerate}
We refer to elements of $X$ as points and elements of $\mathcal{D}$ as \emph{diametrisable sets}. Further, if a function $\diam:\mathcal{D} \rightarrow \mathbb{R}_{\geq0}$ is such that given $U, V \in \mathcal{D}$ with a nonempty intersection, we have that \[\diam(U) + \diam(V) \geq \diam(U \cup V)\] holds, we call $\diam$ the \emph{diameter} and $(X, \mathcal{D}, \diam)$ a \emph{diameter space}. We refer to this inequality as the \emph{triangle inequality for the diameter spaces}.\\

In a very natural way, one can use a diameter space to induce a pseudometric on the underlying space, by simply finding the infimum of diameters of all the diametrisable sets containing any two given distinct points. Furthermore, by imposing suitable conditions on the diametrisable sets, one can get nice properties to hold for the pseudometric space.\\
In order to continue further, we must first specify the underlying set. Hence, let us look for the space that should, in some vague sense, be the minimal counterexample. One of the possible ways to approach this issue is to fix a point $x_0$, and then look what other points can we obtain. It is not hard to see that completion of the set of all images that one can get by applying $f$ and $g$ to $x_0$ is itself a compact metric space, and that $f,g$ form a $\lambda$-contractive map on this subspace of $X$ as well. Now, starting from $x_0$ we must include all the points described, and (provided we cannot get to the same point in different ways) we can actually biject all these points that are obtained using $f,g$ from $x_0$ with finite words over a two letter alphabet. Therefore, our construction will start from an underlying set $X$ of all finite words over $\{a,b\}$, with obvious functions $f,g$ each of which adds one of the characters to the beginning of the word given as input. Then, provided we have a metric on $X$, we will take its completion, and hope that the metric space that we get, along with $f,g$ satisfies all the properties of the Theorem~\ref{thm1}. This is where diameter spaces come into play. We will describe the properties of a collection of diametrisable sets $\mathcal{D}$ on $X$, that guarantee the completion of the induced pseudometric space, along with $f,g$, (being the concatenation functions described), to be a counterexample. Being a relatively long list, we refer the reader to the Corollary~\ref{sumCor} to get the idea of what conditions we impose on $\mathcal{D}$, which will in fact be a sequence of subsets of $X$, denoted by $S_0, S_1, S_2, \dots$.\\

\section{Inducing a counterexample from a diameter space}
\subsection{Diameter spaces and their connection with metric spaces}
In this subsection we show how one can obtain a pseudometric space from diameter spaces. The following proposition tells us how to induce pseudometric on the underlying set.
\begin{proposition} Let $(X, \mathcal{D}, \diam$) be a diameter space. Define a function $d:X^2 \rightarrow \mathbb{R}$ by $d(x, y) = \inf \diam(U)$ when $x$ and $y$ are distinct, where infimum is taken over all diametrisable set that contain points $x$ and $y$, and $d(x,y) = 0$ otherwise. Then $d$ is well-defined and $(X, d)$ is a pseudometric space. \end{proposition}
\begin{proof} Firstly, suppose we are given two distinct points $x$ and $y$. Then $S$, the set of all values that diameter of a diametrisable set containing $x, y$ can take, is non-empty, and bounded from below, so $\inf S$ exists, and $d$ is well-defined.\\
To prove that $d$ is a pseudometric, we just need to show that the triangle inequality holds, since $d(x,x) = 0$ holds for all points $x$ and $d$ is symmetric by construction. Therefore let $x, y, z$ be three points. If any of these are equal, we are done. Otherwise, given $\epsilon > 0$ we can find sets $U, V \in \mathcal{D}$ such that $x, y \in U$, $y, z \in V$, $d(x, y) \leq \diam U \leq d(x, y) + \epsilon/2, d(y, z) \leq \diam V\leq d(y, z) + \epsilon / 2$. But $U$ and $V$ intersect hence their union forms a diametrisable set, and further $\diam( U \cup V) \leq \diam (U) + \diam (V) \leq d(x,y) + d(y, z) + \epsilon$, so $d(x, z) \leq d(x,y) + d(y, z) + \epsilon$. But this holds for any positive $\epsilon$ proving the triangle inequality, and therefore the proposition.\end{proof}
The pseudometric that we constructed from the diameter space will be referred to as the \emph{induced pseudometric} by $\diam$.\\
Remark that given any metric space $(X, d)$ we can construct a diameter space $(X, \mathcal{D}, \diam)$ by taking diametrisable sets to be the finite subsets of $X$, and the diameter function $\diam$ to have the usual meaning, that is for any $U \in \mathcal{D}$, we set $\diam(U)$ to be $\max (x, y)$ where maximum is taken over all pairs of points in $U$. Then, the metric $d$ coincides with the pseudometric induced by $\diam$. Thus in this way, we can consider every metric space to be a diameter space at the same time.\\
In addition we note that any finite measure space is a diameter space as well.
\subsection{Required properties of the diameter space}
In the previous subsection we saw how to obtain a pseudometric space from a diameter space. To construct a counterexample to the conjecture , we will use this procedure, but as the proposition only guaranties that we get a pseudometric space, we need to add additional properties of a diameter space to ensure that we reach our goal. First of all, we work with specific underlying set $X$ and the diameter function which are consistent with the nature (or more precisely the geometry) of the problem. As we are considering essentially compositions of functions applied to an element of the given metric space satisfying the assumptions of the conjecture, we will work with points of $X$ being finite words over the alphabet consisting of two letters $\Sigma = \{a,b\}$, including the empty word.  Before we proceed further, let us introduce some notation.\\

\textbf{Notation.} If $u, v$ are two words of $X$, we write $uv$ for the word obtained by writing first $u$ then $v$. Say that $u$ is a \emph{prefix} or an \emph{initial segment} of $v$, if there is another word $w$ with $v = uw$, and if this holds write $u \leq v$. The \emph{length} of a word $u$, denoted by $l(u)$ is the number of letters in $u$. Characters are considered to be words of length one simultaneously as being characters. When $S$ is a subset of $X$ we write $uS = \{us: s \in S \}$ and $Su = \{su:s\in S \}$, when $u \in X$. Given a positive integer $n$ and word $u$, write $u^n$ for $uu\dots u$, where $u$ appears $n$ times. The empty word is denoted by $\emptyset$. Finally, we allow infinite words occasionally (although these are not elements of $X$), and write $u^\infty$ to stand for the infinite word obtained by writing consecutive copies of $u$ infinitely many times.\\

We will take our functions $f, g:X\rightarrow X$ to be given by $f(u) = au$ and $g(u) = bu$ for all words $u \in X$. Then every word is actually the same as the corresponding composition of functions applied to the empty word. On the other hand, having in mind the contraction property of the family of the functions that we want to hold, we take $\mathcal{D}$ to be a sequence of subsets of $X$, namely $\mathcal{D} = \{S_0, S_1, S_2, \dots\}$ and set $\diam(S_k) = \lambda^k$ for some fixed $\lambda \in (0,1)$. Now, we just need to specify what needs to hold for $\mathcal{D}$ so that we get a counterexample.\\
Let us start with ensuring that $(X, \mathcal{D}, \diam)$ is a diameter space. We accomplish this by requiring the following property:\\

\textbf{A1.} For any nonnegative integers $i < j$ if $S_i \cap S_j \not= \emptyset$ then $S_j \subset S_i$ and $S_0 = X$.\\

Then we see that $(X, \mathcal{D})$ is a diametrisable space and that given any two diametrisable sets $S_i, S_j$ that intersect, we have $S_i \cup S_j = S_i$ or $S_i \cup S_j = S_j$ so the triangle inequality holds, thus this is indeed a diameter space.\\

Consider now $(X, d)$ where $d$ is the induced pseudometric by $\diam$. To make $d$ non-degenerate, we introduce another property:\\

\textbf{A2.} Each point belongs to only finitely many diametrisable sets.\\

Since $\diam$ is always positive, and the infimum defining $d(x,y)$ for $x\not=y$ is actually minimum taken over finitely many positive values, we get that $d$ is nondegenerate, and thus a metric.\footnote{It is easy to see that this is in fact an ultra metric.}\\

As far as the compactness is concerned, the fact that we can obtain a compact space from a totally bounded one by taking its completion is what motivates our following step. Hence, another condition is introduced:\\

\textbf{A3.} For any positive integer $N$, there are integers $i_1, i_2, \dots,i_n$ greater than $N$ such that $X \setminus (S_{i_1}\cup S_{i_2}\cup \dots\cup S_{i_n})$ is finite.\\

\begin{proposition} The completion $(\overline{X}, d)$ of the metric space induced by $(X, \mathcal{D}, \diam)$ is compact, provided $\mathcal{D}$ satisfies \textbf{A1}-\textbf{A3}.\end{proposition}
\begin{proof} By the comments from the above, we just need to show that $(X, d)$ is totally bounded. Let $\epsilon > 0$ be given, and choose $N$ for which $\lambda^N < \epsilon/2$ holds. By \textbf{A3}, there are $i_1, i_2, \dots, i_n > N$ for which the union of $S_{i_1}, S_{i_2}, \dots, S_{i_n}$ covers all but finitely many points, denoted by $y_1, y_2, \dots, y_m$. Then, each $S_{i_k}$ is contained in $B_{x_k}(\epsilon)$ for some $x_k \in S_{i_k}$ and $y_k \in B_{y_k}(\epsilon)$, so $X$ is covered by finitely many open balls of radius $\epsilon$, and so the metric space is totally-bounded.\end{proof}
Say that a Cauchy sequence $(x_n)_{n\geq 1}$ is proper if there is no $N$ with $x_N = x_{N+1} = \dots$. The three conditions described so far give us a nice characterisation of the proper Cauchy sequences in $(\overline{X}, d)$, whose elements lie in $X$.
\begin{proposition}\label{cauchychar}Suppose that $\mathcal{D}$ satisfies properties \textbf{A1}-\textbf{A3}. Then a sequence of points in $X$ is proper Cauchy with respect to the induced metric if and only if for any given positive integer $M$ there is $m > M$ such that $S_m$ contains all but finitely many points of the sequence. \end{proposition}
\begin{proof} \emph{Only if direction.} Let $(x_n)_{n\geq1}$ be a proper Cauchy sequence in $X$ and let a positive integer $M$ be given. Take a positive $\epsilon < \lambda ^ M$. Then, as the given sequence is Cauchy, we have $N$ such that $m,n > N$ implies $d(x_n, x_m) < \epsilon$.\\
Now fix any $m > N$ and let $I = \{i > M:\exists n > N, x_n, x_m \in S_i\}$. By the definition of $d$, we know that this set is nonempty, therefore has a minimal element $i_0$. If $n > N$ then $d(x_n, x_m) < \epsilon < \lambda^M$, so there is $j > M$ with $x_n, x_m$ both belonging to $S_j$. But $x_m \in S_{i_0}$ so $S_{i_0}, S_j$ intersect and by the choice of $i_0$ we have $x_n \in S_j \subset S_{i_0}$, so almost all points of the sequence are contained in $S_{i_0}$.\\
\emph{If direction.} Given $\epsilon > 0$ take $M$ such that $\lambda ^ M < \epsilon$. Then there is $m > M$ with $S_m$ containing almost all points in the sequence, and the distance between two points in $S_m$ is at most $\lambda ^ {M+1} < \epsilon$, so the sequence is Cauchy. If it was not a proper one, the point which is equal to almost all of its members would belong to infinitely many of the sets $S_i$ which is impossible.\end{proof}
\begin{proposition} Under the same assumptions as in the previous proposition, no proper Cauchy sequence in $(\overline{X}, d)$ converges to a point in $X$.\end{proposition}
\begin{proof} Suppose that a point $x \in X$ is a limit of a proper Cauchy sequence $(x_n)$ in $\overline{X}$. First of all, if $x_n \in \overline{X}\setminus X$, substitute $x_n$ by a point $y \in X$ such that $d(x_n, y)< \frac{1}{2}d(x_n,x)$. The newly obtained sequence now lies in $X$, and is still proper Cauchy, with unchanged limit. As $x \in X$ belongs to $S_i$ for only finitely many $i$, we have $B_{x}(\lambda^n) \cap X = \{x\}$, for sufficiently large $n$. This is however a contradiction, due to the Proposition~\ref{cauchychar}.\end{proof}
\begin{lemma} Let $(\overline{X}, d)$ be the metric space considered so far. Suppose $F$ is a function on $X$ that preserves Cauchy sequences, that is given Cauchy sequence $(x_n)_{n\geq1}$, the sequence of images $(F(x_n))_{n\geq1}$ is Cauchy as well. Then, extension of $F$ to the completion of the space given by $F(x) = \lim_{n\rightarrow \infty} F(x_n)$, where $(x_n)$ is any Cauchy sequence in $X$ tending to $x$ in $\overline{X} \setminus X$, is continuous.\end{lemma}
\begin{proof}Firstly, we should prove that such extension of $F$ is well-defined. Take arbitrary $x \in \overline{X} \setminus X$, and thus since $\overline{X}$ is completion of $x$, there must be Cauchy sequence in $X$ whose limit is $x$. But image of this sequence under $F$ is Cauchy as well, so it has limit in $\overline{X}$, so we just need to show its uniqueness. Therefore, suppose that $(x_n)_{n\geq1}, (y_n)_{n\geq1}$ are two sequences in $X$ tending to $x$. Merging these two sequence into $(t_n)_{n\geq1}$, where $t_{2n-1} = x_n, t_{2n} = y_n, n \geq 1$, implies that $(t_n)_{n\geq1}$ is Cauchy, hence $(F(t_n))_{n\geq1}$ is also Cauchy, so $(F(x_n))_{n\geq1}$ and $(F(y_n))_{n\geq1}$ have the same limit, as required.\\
Secondly, we should prove that $F$ is continuous in $\overline{X}$. Let $(x_n)_{n\geq1}$ be sequence tending to some $x \in \overline{X}$. If $x \in X$, then sequence is eventually constant and equal to $x$, hence $(F(x_n))_{n\geq1}$ trivially tends to $F(x)$. Otherwise, $x \not \in X$, so consider new sequence $(t_n)_{n\geq1}$ given as follows. If $x_n \in X$, set $t_n = x_n$, and if this does not hold, there is Cauchy sequence $(y_m)_{m\geq1}$ in $X$ whose limit is $x_n$. By assumption, $(F(y_m))_{m\geq1}$ is Cauchy in $X$, and as we have shown previously, it tends to $F(x_n)$. Hence, for sufficiently large $m$, we have that $d(x_n, y_m), d(F(x_n), F(y_m)) < 1/n$, so set $t_n = y_m$. Thus, as for all $n$ we have $d(x_n, t_n) < 1/n$, we have that $(t_n)_{n\geq1}$ is Cauchy in $X$ and tends to $x$, so its image under $F$ is Cauchy sequence with limit $F(x)$, but $d(F(x_n), F(t_n)) < 1/n$ holds for all $n$, thus $\lim_{n\rightarrow \infty} F(x_n) = F(x)$, as required, implying continuity of $F$.\end{proof}
This lemma suggests the fourth property of the diametrisable sets:\\

\textbf{A4.} If $i_1 < i_2 < \dots$ are indices such that $S_{i_1} \supset S_{i_2} \supset \dots$ then, given any $N$, we can find $n_a, n_b > N$ for which $S_{n_a}$ contains all but finitely many elements of $aS_{i_m}$ for some $m$, and $S_{n_b}$ contains all but finitely many elements of $bS_{i_p}$ for some $p$.\\

\begin{proposition} If $\mathcal{D}$ satisfies \textbf{A1}-\textbf{A4}, then $f, g:\overline{X} \rightarrow \overline{X}$, defined before and then extended as in the previous lemma are continuous with respect to the induced metric.\end{proposition}
\begin{proof} We will show the claim for $f$, proof for $g$ follows the same lines. We only need to show that $f$ preserves Cauchy sequences in $X$, in fact it is sufficient to prove that if $(x_n)_{n\geq1}$ is proper Cauchy, then $(f(x_n))_{n\geq1}$ is Cauchy. Thus, suppose we are given a proper Cauchy sequence $(x_n)_{n\geq1}$ in $X$, so there indices $i_1 > 1, i_2 > 2, \dots$ (without loss of generality $i_1 < i_2 < \dots$) such that $S_{i_k}$ covers all but finitely many elements of the sequence for every $k$. Due to intersections and \textbf{A1} we have $S_{i_1} \supset S_{i_2} \supset \dots$. Let $\epsilon > 0$ be given and take $N$ which satisfies $\lambda^N < \epsilon$. Further, by \textbf{A4} there is $n > N$ and some $m$ for which $S_n$ contains all but finitely many elements of $aS_{i_m}$. Exploiting the fact that almost all elements of $(f(x_n))_{n\geq1}$ are contained in $aS_{i_m}$ yields that there is $M$ such that $k,l > M$ implies $f(x_k), f(x_l) \in S_n$ hence $d(x_k, x_l) \leq \lambda^N < \epsilon$, as required.\end{proof}
To make $\{f,g\}$ a contractive family, we define the property:\\

\textbf{A5.} For any $i \in \{0, 1, 2, \dots\}$, there is $j > i$ such that $c S_i \subset S_j$ for some character $c\in\{a,b\}$.\\

\begin{proposition} If $\mathcal{D}$ satisfies \textbf{A1}-\textbf{A5}, then $f,g$ form a $\lambda$-contractive family in $(\overline{X},d)$.\end{proposition}
\begin{proof} Let us consider first $x, y \in X$. Taking largest possible $n$ for which $x, y \in S_n$, we have $m > n$ by \textbf{A5} such that without loss of generality $aS_n \subset S_m$, thus $d(f(x), f(y)) = d(ax, ay) \leq \lambda^m \leq \lambda \lambda^n = \lambda d(x,y)$, as wanted. (If we had character $b$ instead, we would get contraction when applying $g$.)\\
In general case, $x, y \in \overline{X}$, we can find sequences in $X$, $(x_n)_{n\geq 1}$ tending to $x$, and $(y_n)_{n\geq1}$ tending to $y$ (if one of these is already in $X$, then take trivial sequence). Then, by previous case, one of $f,g$ contracts infinitely many pairs $(x_n, y_n)$, $f$ say. Let indices of those pairs be $i_1 < i_2 < \dots$. Then, $\lambda d(x, y) = \lambda \lim_{n\rightarrow \infty} d(x_{i_n}, y_{i_n}) \geq \lim_{n\rightarrow \infty} d(f(x_{i_n}), f(y_{i_n})) = d(f(x),f(y))$, since $f$ is continuous, as required. \end{proof}
\begin{proposition} Suppose that a function $F$, which is a word in $f,g$, has a fixed point. Let $w$ be the nonempty word which corresponds to $F$, i.e. $F(x) = wx$ for all $x$. Then there are infinitely many $i$ such that there is $u_i$ in $S_i$ with $wu_i$ also being a member of $S_i$.\end{proposition}
\begin{proof} Suppose $F$ fixes $u$. Then $u \not \in X$, so take $(x_n)_{n\geq1}$ in $X$ converging to $u$. Hence $(wx_n)_{n\geq1}$ converges to $u = F(u)$ as well, so merging these two sequences together, we get a Cauchy sequence, and the result follows from the Proposition~\ref{cauchychar}.\end{proof}
Therefore, let us introduce the final property of the diametrisable sets that we need to hold. The following corollary does this and sums up our work so far.
\begin{corollary}\label{sumCor}Let $\lambda \in (0,1)$. Consider a collection $\mathcal D$ of diametrisable subsets of $X$ that obeys:\begin{description}
\item[\textbf{A1}] The set $S_0$ is the whole of $X$ and given any nonnegative integers $i < j$, $S_i \cap S_j \not= \emptyset$ implies $S_j \subset S_i$.
\item[\textbf{A2}] Any point in $X$ belongs to only finitely many diametrisable sets.
\item[\textbf{A3}] For any positive integer $N$, there are integers $i_1, i_2, \dots,i_n$ greater than $N$ such that $X \setminus (S_{i_1}\cup S_{i_2}\cup \dots\cup S_{i_n})$ is finite.
\item[\textbf{A4}] If $i_1 < i_2 < \dots$ are indices such that $S_{i_1} \supset S_{i_2} \supset \dots$ then, given any $N$, we can find $n_a, n_b > N$ for which $S_{n_a}$ contains all but finitely many elements of $aS_{i_m}$ for some $m$, and $S_{n_b}$ contains all but finitely many elements of $bS_{i_p}$ for some $p$.
\item[\textbf{A5}] For any $i \in \{0, 1, 2, \dots\}$, there is $j > i$ such that $c S_i \subset S_j$ for some character $c\in\Sigma$.
\item[\textbf{A6}] Given a nonempty word $w$ in $X$ there are only finitely many diametrisable sets $S_i$ with $u_i \in S_i$ for which $wu_i \in S_i$.
\end{description}
Then, with constructions described above, $(\overline{X}, d)$ is a compact metric space with continuous functions $f, g: \overline{X} \rightarrow \overline{X}$ that form a $\lambda$-contractive family, but no word in $f,g$ has a fixed point.
\end{corollary}

\section{Choosing the diametrisable sets}
Our main task now is to find a collection of diametrisable sets $\mathcal{D}$ which has the properties \textbf{A1}-\textbf{A6}. The most natural subsets of $X$ to be examined are $W_w = wX$, where $w$ is any word in $X$. Immediately, we observe that these, ordered by length of $w$, then by alphabetical order (so that $S_0 = W_\emptyset, S_1 = W_a, \dots$) satisfy all the properties except for \textbf{A6}. Hence, we will use these as the pillar of our construction, however, to make \textbf{A6} hold, we need to modify these slightly. The issue with sets described is that given any nonempty word $w \in X$, we allow $W_{w}$ to contain all the initial segments of $w^\infty$. With this in mind, we say that a nonempty word $w$ in $X$ is forbidden if there is another finite word $u$ such that $w$ is a prefix of $u^\infty$ and $l(w) > l(u)^2$. Otherwise, say that $w$ is available. For example, $\emptyset, a, abab$ are available, while $ababa$ is forbidden.
\begin{proposition}\label{addavailProp}Given a word $w \in X$ either $aw$ or $bw$ is available. \end{proposition}
\begin{proof}  If $w = \emptyset$, $aw$ is available. Suppose $l(w) \geq 1$ and that the claim is false, so $aw$ is an initial segment of $w^\infty_1$ for some non-empty $w_1$, $bw$ is an initial segment of $w_2^\infty$ for some non-empty $w_2$ and $l(w) \geq l(w_1)^2, l(w_2)^2$. We can permute $w_1$ cyclically to $v_1$ so that $av_1^{\infty} = w_1^\infty$ holds, and we can correspondingly transform $w_2$ to $v_2$. Observe that the last character of $v_1$ is $a$ and of $v_2$ is $b$. This way, $w$ becomes a prefix of $v_i^{\infty}$ for $i = 1,2$. But $l(w) \geq l(v_1)^2, l(v_2)^2$, hence $l(w) \geq l(v_1) l(v_2)$ and so $v_1^{l(v_2)} = v_2^{l(v_1)}$, by comparing them as initial segments of $w$. However, this implies equality of the last characters of $v_1$ and $v_2$, which is impossible. The claim now follows. \end{proof}
In the same spirit, we prove the following statement:
\begin{proposition}Given a word $w \in X$ either $wa$ or $wb$ is available. \end{proposition}
\begin{proof} Suppose contrary, both are forbidden. Let $t_1, t_2$ be words such that $wa\leq t_1 ^\infty$, $wb\leq t_2^\infty$, $l(wa) > l(t_1)^2, l(wb) > l(t_2)^2$. Observe that both $t_1^{l(t_2)}$ and $t_2^{l(t_1)}$ are prefixes of $w$, so being of the same length, they coincide. But, then $t_1^\infty = t_2^\infty$, which is a contradiction, as otherwise $wa \not= wb$ would be two prefixes of the same length of this infinite word.\end{proof}

\begin{corollary} \label{availCor}Let $w$ be a non-empty word. Let $u$ be an initial segment of $w^\infty$, and suppose $l(u) \geq l(w)^2$. Take character $s$ such that $v = us$ is not a prefix of $w^\infty$. Then $v$ is available. \end{corollary}
Say that a word $w$ is minimal if it is non-empty and given non-empty $u$ for which $u^\infty = w^\infty$ we have $l(w) \leq l(u)$.
\begin{proposition} A non-empty word $w$ is not minimal iff there is word $u$ such that $w = u^k$, some $k \geq2$.\end{proposition}
\begin{proof} If $w = u^k$ with $k > 1$ then $l(w) > l(u)$ and $w^\infty = u^\infty$, therefore $w$ is not minimal. Suppose now that we have non-empty $w$, for which there exists $u$ such that $u^\infty = w^\infty$, but $l(u) < l(w)$. Write $d = \gcd(l(u), l(w))$, so $l(u) = qd, l(w) = pd$, for some positive integers $p, q$, in particular $p \geq 2$. Further $u^p = w^q = v_1v_2\dots v_{pq}$, where $v_1, v_2, \dots, v_{pq}$ are of length $d$. Considering successive copies of $u$ we have that $v_{i+q} = v_i$, when $i \leq pq - q$ and similarly by looking at copies of $w$ we have $v_{i+p} = v_i$ when $i \leq pq - p$. So $v_i = v_{p+i} = \dots = v_{p(q-1) + i}$ for $i \in[p]$, (where for a positive integer $N$, $[N]$ denotes the set $\{1,2, \dots, N\}$). Observe that $v_q = v_{2q} = \dots = v_{pq}$ and as $p$ and $q$ are coprime $q, 2q, \dots, pq$ take all possible values modulo $p$ hence $v_1 = v_2 = \dots = v_{pq}$, allowing us to conclude $w = v_1 ^ p$, $p > 1$. \end{proof}
Having established these results about the words, we are ready to choose the diametrisable subsets $\mathcal{D}$ of $X$. Consider the following subsets.\\

For all available words $w$, including $\emptyset$, we include $W_w$ in $\mathcal{D}$ (we refer to these as \emph{W-type sets}, i.e say that $S_i$ is of $W$-type if $S_i = W_w$, for some available $w$).\\

For all minimal words $w$, all integers $p, r$ such that $p \in \{2^i: i \in \mathbb{N}_0\}$ and $0 \leq r \leq p-1$, we set $A_{w,p,r} = \{u \in X: u$ \vphantom{a}initial segment of \vphantom{a}$w^\infty, l(u) \in \{r+ip: i \in \mathbb{N}\}, l(u) > l(w)^2\}$ (call these \emph{$A$-type} sets).\\

Finally, for each minimal word $w$ and $k \in \mathbb{N}$, we define $B_{w,k}$, which we refer to as \emph{$B$-type} sets. For these, we need additional notation.\\
First of all, for each $k \in \mathbb{N}$ we define an infinite arithmetic progression $I_k$. We set $I_1 = \mathbb{N}$, $I_2 = I_1 \setminus \{\min I_1\} = \mathbb{N} \setminus\{1\}$. For each integer $m \geq 2$, if $k$ is an integer such that $2^m - 1 \leq k \leq 2^{m} + 2^{m-1} - 2$, we set $I_k = \{s + i.2^{m-1}: i \in \mathbb{N}_0\}$ where $s = \min I_{\frac{k-1}{2} + 2^{m-2}}$, when $k$ is odd; if $k$ is even, then put $I_k = \{s + i.2^{m-1}:i \in \mathbb{N}_0\}$ where $s = \min I_{\frac{k-2}{2} + 2^{m-2}} + 2^{m-2}$. On the other hand, if $k$ is an integer such that $2^{m} + 2^{m-1} - 1 \leq k \leq 2^{m+1} - 2$, then we set $I_k = I_{k - 2^{m-1}} \setminus \{\min I_{k - 2^{m-1}}\}$.\\
Note that, given a minimal word $w$ and $n \in \mathbb{N}$, by the Corollary~\ref{availCor} we get a unique word $w_n$ such that: it is of the form $v_n s_n$ where $v_n \leq w^\infty$ and $s_n$ is a character, $w_n$ is available, and $l(w_n) = l(w)^2 + n$. At last, we define $B_{w,k} = \cup_{i \in I_k} W_{w_i}$.\\

Thus, we set $\mathcal{D} = \{W_w: w \in X, w$ is available$\} \cup \{A_{w, p, r}: w \in X, p, r \in \mathbb{Z}, w$ is minimal, $p \in \{2^i: i \in \mathbb{N}_0\}, 0 \leq r \leq p-1\} \cup \{B_{w,k}: w \in X, k \in \mathbb{N}, w$ is minimal$\}$.\\

To illustrate the definition of $B$-type sets, we list a few examples: $I_{7} = \{4,8,12,16,\dots\}, I_{13} = \{9,13,17,21, \dots\}, B_{a, 7} = W_{aaaab} \cup W_{aaaaaaaab} \cup W_{aaaaaaaaaaaab} \cup \dots, B_{ba, 13} = W_{babababababaa} \cup W_{babababababababaa} \cup W_{babababababababababaa} \cup \dots$. 

Let us make a few easy remarks about $\mathcal{D}$. Fix a minimal word $w$. Then we have $B_{w,1} = W_{w^{l(w)}} \setminus (A_{w,1,0} \cup \{w^{l(w)}\})$. Also, if $m \geq 2$ is an integer, and $k$ is an odd integer that satisfies $2^m-1 \leq k \leq 2^m + 2^{m-1} -2$, then $I_{\frac{k-1}{2} + 2^{m-2}} = I_{k} \cup I_{k+1}$. Furthermore, if $I_{k_1} \cap I_{k_2} \not=\emptyset$ then $I_{k_1} \subset I_{k_2} $ or vice-versa.\\

%\begin{itemize}
%\item $B_{w, 1} = W_{w^{l(w)}} \setminus (A_{w,1,0} \cup\{w^{l(w)}\}) = \cup^\infty_{n=1}W_{w_n}$, where $w_n = v_n s_n$, $v_n$ being prefix of $w^\infty$, $s_n$ character chosen to make $w_n$ available by Corollary~\ref{availCor}, and $l(w_n) = l(w)^2 + n$.
%\item Consider the following tree, shown in the Figure \ref{Figure1}, where all nodes in odd layers have one son and those in even have two.
%\begin{figure}
%\includegraphics[scale = 0.4]{Btype.eps}
%\caption{The auxiliary tree explaining the choice of the diametrisable sets}
%\label{Figure1}
%\end{figure}
%We will have all $B_{w,k}$ to be $\cup_{i \in I_k} W_{w_i}$, where $I_k$ is arithmetic progression of common difference which is a power of 2, and the words $w_i$ are those appearing in the definition of the set $B_{w,1}$. If node $k$ has sons $s_1, s_2$ and $I_k$ has common difference $2^\alpha$, we define $B_{w,s_1} = \cup_{i\in I_{s_1}} W_{w_i}, B_{w,s_2} = \cup_{i\in I_{s_2}} W_{w_i}$, where $I_{s_1}$ and $I_{s_2}$ are the two arithmetic progressions with steps $2^{\alpha+1}$ which cover $I_k$. For completeness, if $s_1 < s_2$ then $\min I_{s_1} < \min I_{s_2}$. If $k$ has one son, say $s$, then we take $I_s = I_k \setminus \{\min I_k\}$ (which is arithmetic progression of same common difference as $I_k$). (Call these \emph{$B$-type} sets.)
%\end{itemize}
For any $U \in \mathcal{D}$, observe that there is a unique word in $U$ of the shortest length, which we will denote by $\sigma(U)$.\\
Let us now establish a few claims about the structure of $\mathcal{D}$, which will be exploited in the rest of the proof. 
\begin{proposition} \label{containProp}If $U\not= V$ are two diametrisable sets and they intersect, then one is contained in the other. Furthermore, if they are not identical, then $U \triangle V$ is infinite.\end{proposition}
\begin{proof} We are going through possible types of $U$ and $V$.
\begin{enumerate}
\item Case 1: $U$ and $V$ are $W$-type.\\
Suppose $U = W_{w_1}, V = W_{w_2}$ without loss of generality $l(w_1) \leq l(w_2)$. If $w_1$ is not initial segment of $w_2$ then $U$ and $V$ do not intersect, so we must have $W_{w_1} \supset W_{w_2}$, that is $U \supset V$. If $U \not= V$, then $w_1 \not= w_2$, and hence $W_{w_1 s} \subset (U \setminus V)$, where $s$ is character for which $w_1 s$ is not prefix of $w_2$. We will use this case for showing the other ones.
\item Case 2: $U$ and $V$ are $A$-type.\\
Suppose $U = A_{w_1, p_1, r_1}, V = A_{w_2, p_2, r_2}$, where $w_1, w_2$ are some minimal words, and $p_1, p_2, r_1, r_2$ are suitable integers. There is word $w \in U \cap V$, so $w$ is initial segment of $w_1^\infty$ and $w_2^\infty$, while $l(w) \geq l(w_1)^2, l(w_2)^2$, from which we deduce that $w_1^{l(w_2)} = w_2^{l(w_1)}$, being the initial segment of $w$ of length $l(w_1) l(w_2)$. Therefore $w_1^\infty = w_2^\infty$ and due to minimality $w_1 = w_2$. Now, due to definition of $A$-type sets for fixed minimal word, we get $U \subset V$ or vice versa and infinite symmetric difference.
\item Case 3: $U$ is $A$-type, $V$ is $B$-type. \\
Let $w\in U \cap V$. This makes $w$ a prefix of some $t^\infty$, where $t$ is minimal and $l(t)^2 < l(w)$, as $U = A_{t,p,r}$, $p$, $r$ being suitable integers. Also $w\in W_{v} \subset V$, some available $v$, so as $v$ is an initial segment of $w$, hence $t^\infty$, we have $l(v) \leq l(t)^2$, since $v$ is available. But then $U \subset A_{t, 1, 0} \subset W_{t^{l(t)}} \subset W_v \subset V$, which proves the first claim.\\
On the other hand $U$ has no available words, but $V$ has infinitely many of these, which gives the second part.
\item Case 4: $U$ is $A$-type, $V$ is $W$-type. \\
Same proof as in the Case 3.
\item Case 5: $U$ and $V$ are $B$-type. \\
Suppose $U = B_{w_1, k_1}, V = B_{w_2, k_2}$, for some minimal words $w_1, w_2$ and positive integers $k_1, k_2$. Let $w \in U \cap V$, thus $w \in W_{v_1} \subset B_{w_1, k_1}$ and $w \in W_{v_2} \subset B_{w_2, k_2}$. By Case 1, without loss of generality, $W_{v_1} \subset W_{v_2}$ holds.\\
If $v_1 = v_2$, letting $u_i$ be $v_i$ without the last character, gives us prefix of $w_i^\infty$, $i \in [2]$, and $l(w_1)^2, l(w_2)^2 \leq l(u_1)$, so as before $w_1^\infty = w_2^\infty$ and due to minimality $w_1 = w_2$. So $U \subset V$ or $V \subset U$, due to construction of $B$-type sets for a fixed minimal word, and also the second part of the claim follows.\\
On the other hand, if $v_1 \not= v_2$, then $v_2\leq u_1$, $u_1$ being $v_1$ after omitting the last character, as before. Further we have that $u_1$ is a prefix of $w_1^\infty$, so $v_2$ is too, but $v_2$ is available, hence $l(v_2) \leq l(w_1)^2$, hence $U \subset W_{v_2} \subset V$.\\
For the second part of the claim, consider any other $W$-type set contained in $V$, distinct form $W_{v_2}$.
\item Case 6: $U$ is $B$-type, $V$ is $W$-type. \\
Let $V = W_v$, some available word $v$, and $w \in U \cap V$, so $w \in W_u \subset U$. If $W_v \subset W_u$, then $V \subset U$, so we are done, and second part follows as in previous case. Otherwise, by Case 1, $W_u \subsetneq W_v$, so $v \not = u$ and $v\leq u$. We have $U = B_{t,k}$, some minimal $t$ and integer $k$, so $v\leq t^\infty$ and available so $l(v) \leq l(t)^2$, hence $U \subset W_v = V$. Also, $(W_v \setminus U) \supset W_{t^{l(t)}} \setminus B_{t, 1} = A_{t, 1} \cup \{t^{l(t)}\}$, which is infinite.
\end{enumerate}
\end{proof}
\begin{proposition} If $U$ is any of the diametrisable sets, then there are unique $V_1, V_2$ (up to ordering) proper subsets of $U$ in $\mathcal{D}$, such that $V_1 \cup V_2$ has almost all elements of $U$.\end{proposition}
\begin{proof} Firstly, let us show the existence of such sets. Of course, we go through the possible types of set $U$. The only non-trivial case, that is the one that does not directly follow from construction of $\mathcal{D}$ is the third one.
\begin{enumerate}
\item Case 1: $U$ is of $A$-type.\\
Say $U = A_{w, p, r}$, some minimal $w$ and integers $p, r$. Then, we can take $V_1 = A_{w, 2p, r}, V_2 = A_{w, 2p, p + r}.$
\item Case 2: $U$ is of $B$-type.\\
Say $U = B_{w, k}$, some minimal $w$ and integer $k$. Then, by construction, we have either $B_{w, k} = B_{w, k_1} \cup B_{w, k_2}$, some two integers $k_1, k_2$, or $B_{w,k} = B_{w,k+1} \cup W_{v}$, some available $v$, giving us the needed sets $V_1, V_2$.
\item Case 3: $U$ is of $W$-type.\\
Suppose $U = W_{w}$, for some available $w$. Then, if $wa$ and $wb$ are both available, we can take $V_1 = W_{wa}$ and $V_2 = W_{wb}$.\\
Otherwise, suppose this is not true, and $ws$ is not available for some character $s$. Therefore, there is a word $t$ such that $ws$ is prefix of $t^\infty$, and $l(t)^2 < l(ws) = l(w) + 1$. But $w$ is also a prefix of $t^\infty$ and available so $l(t)^2 \geq l(w)$ so $l(w) = l(t)^2$, i. e. $w = t^{l(t)}$.\\
The next thing to do is to establish minimality of $t$. Let $u$ be any word such that $l(u) \leq l(t)$ and $t^\infty = u^\infty$. Then, as before $ws$ is initial segment of $u^\infty$ and $l(u)^2 < l(ws) = l(w) + 1$, so by same arguments we get $l(w) = l(u)^2$, thus $l(u) = l(t)$, as wanted. But then $W_{w} = A_{t, 1, 0} \cup B_{t, 1} \cup \{w\}$, proving the existence part of the claim.\\ \end{enumerate}
Suppose now that $U'$ is any set strictly contained in $U$. Then, by Proposition~\ref{containProp}, $U\setminus U'$ is infinite, so $U'$ cannot contain both $V_1, V_2$. Also, from the same proposition we get that $V_1$ and $V_2$ are disjoint, as otherwise, one must contain the other and thus be equal to $U$. So, $U'$ intersects at precisely one of $V_1, V_2$, since all diametrisable sets are infinite. W.l.o.g. $U'$ intersects $V_1$. If $V_1 \subsetneq U'$, then $U' \setminus V_1$ is infinite, and thus intersects $V_2$, which is impossible. Hence, by Proposition~\ref{containProp}, $U' \subset V_1$. So, if we had any other $V'_1, V'_2$ with the property in assumption, we would have $V'_1 \subset V_1$ and $V'_2 \subset V_2$, reordering if necessary, (cannot have both sets included in the same $V_i$), so unless these are both equalities $V_1 \cup V_2 \setminus (V'_1 \cup V'_2)$ would be infinite, yielding a contradiction, and concluding the proof.
\end{proof}
\begin{corollary}\label{structCor} If $U$ is a diametrisable set with proper diametrisable subsets $V_1, V_2, U'$, such that $U \setminus (V_1 \cup V_2)$ is finite, then $V_1, V_2$ are disjoint and one of them contains $U'$.\end{corollary}
\begin{proposition}\label{finProp} Given word $w$, there are only finitely many diametrisable sets containing it.\end{proposition}
\begin{proof} If $w \in W_u$, some $u$, then $u$ is prefix of $w$, hence there are only finitely many such sets containing $w$.\\
If $w \in A_{t, p, r}$, then, $l(w) \geq p, l(t)$, so there are only finitely many choices for $t, p, r$.\\
Finally, suppose $w \in B_{t,k}$. Length of $w$ must be greater than $l(t)^2$, which gives us finitely many choices for minimal word $t$. Fix $t$. Recalling the definition of $B$-type sets, we have $B_{t,k} = \cup_{i \in I_k} W_{w_i}$, where $l(w_i) = l(t)^2 + i$. But if $m$ is a nonnegative integer, then $\min I_k \geq m$ for $k \geq 2^m - 1$. Thus, for such $k$, we get $l(\sigma(B_{t,k})) \geq l(t)^2 + m$, so $w \in B_{t,k}$ for only finitely many $k$, as desired.\end{proof}
%Recall the tree we used in construction of $B$-type sets, and suppose $m$ is a node with a child $p$. Then, $l(\sigma(B_{t,m})) \leq l(\sigma(B_{t,p}))$, but if $m$ is in an odd layer, then $p$ is its only child, and we have strict inequality. Hence $k$ must be no deeper than $2l(w)$-th layer, hence there is a bound on $k$ itself, so the claim follows.
These claims serve us to understand better the structure of $\mathcal{D}$. In particular, we can view $\mathcal{D}$ as a binary tree whose nodes are the diametrisable sets, root is $W_{\emptyset}$ and given a set $U\in\mathcal{D}$, its sons $V_1, V_2$ are given by the Corollary~\ref{structCor}. What is not clear, however, is to see that the tree so defined actually contains all the diametrisable sets. But, given any such a set $U\in\mathcal{D}$, we have either $U = W_{\emptyset}$ or $U$ is contained in one of the sons of the root, by Corollary~\ref{structCor}. Proceeding further in this fashion, either we reach $U$, or we get an infinite collection of diametrisable sets whose subset $U$ is. But, that implies that if do not reach $U$, its elements belong to infinitely many members of $\mathcal{D}$, which contradicts Proposition \ref{finProp}. Hence this binary tree has precisely $\mathcal{D}$ for its set nodes. Moreover, to say that a diametrisable set $U_1$ is a subset of another such set $U_2$ is equivalent to having $U_1$ as an ancestor of $U_2$ in this binary tree. To depict what has just been discussed, we include the Figure~\ref{treeFig} which shows the first few layers of tree. We refer to this tree as $\mathcal{T}$.
\begin{figure}
\includegraphics[scale = 0.8]{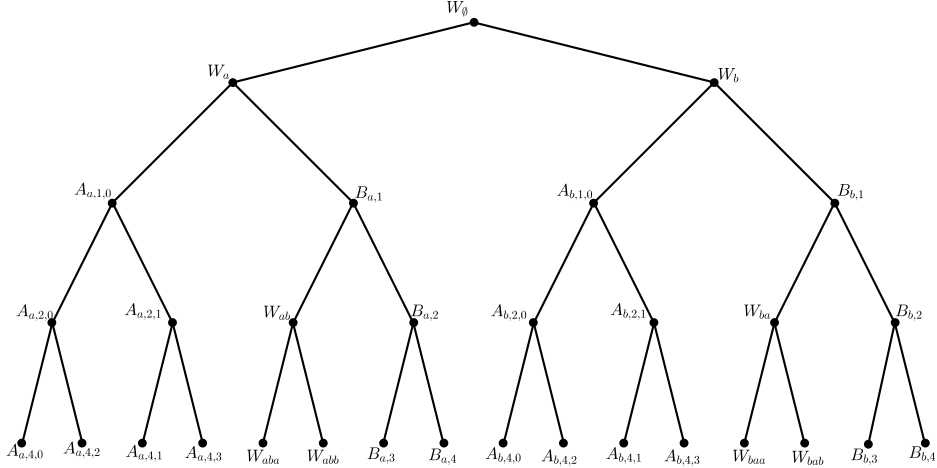}
\caption{Structure of the collection of diametrisable sets}
\label{treeFig}\end{figure}

\section{Ordering $\mathcal{D}$}
In order to finish constructing the counterexample, we must make $\mathcal{D}$ well-ordered, so that then we know which diametrisable set is actually which $S_n$. Defining such an order, and proving that it is in fact what we need, is the purpose of this section.\\
Consider the relation $<$ on $\mathcal{D}$, given as follows: if $U \not= V$, we say $U < V$ if any of these holds:
\begin{itemize}
\item[\textbf{O1}] $l(\sigma(U)) < l(\sigma(V))$
\item[\textbf{O2}] $l(\sigma(U)) = l(\sigma(V))$ and $U \supset V$
\item[\textbf{O3}] $l(\sigma(U)) = l(\sigma(V))$, none is contained in the other and we have that either $U$ is of $A$-type, but V is not, or $U$ is of $B$-type and $V$ is of $W$-type.
\item[\textbf{O4}] $l(\sigma(U)) = l(\sigma(V))$, none is contained in the other, they are of the same type and $\sigma(U)$ is alphabetically before than $\sigma(V)$.
\end{itemize}
\begin{proposition}\label{typeProp}If $U < V$ by \textbf{O2}, then either:
\begin{enumerate}
\item both $U$ and $V$ are $A$-type, or
\item both $U$ and $V$ are $B$-type, or
\item $U$ is $B$-type, $V$ is $W$-type.
\end{enumerate}
\end{proposition}
\begin{proof} Regard $\mathcal{D}$ as a binary tree $\mathcal{T}$ that was described in the concluding remarks of the previous section. Then, $U \subset V$ tells us that $U$ is an ancestor of $V$, and as $U < V$ by \textbf{O2}, we have $\sigma(U) = \sigma(V)$. Thus, the shortest word must same for all sets on the path from $U$ to $V$ in $\mathcal{T}$. Now, we analyze the splits, i.e. given a node, what its sons are. Returning back to the choice of the diametrisable sets, we see that $A$-type sets always split into two $A$-type sets, and $B$-type sets split into $B$ and $W$-type sets. We conclude that the claim will be true if when $W$-type set splits into an $A$-type and a $B$-type set, the shortest word of the father is not in any of the sons. But, suppose that we have such a situation, a set $T = W_w$, whose children are $A_{t, 1, 0}$ and $B_{t,1}$. From the proof the Proposition, recall that $w = t^{l(t)}$. But then $w \notin A_{t,1,0}$ and $w\notin B_{t,1}$, as desired. \end{proof}
\begin{proposition}Relation $<$ on the chosen sets is strict total order, and further collection of chosen sets is well ordered under $<$.\end{proposition}
\begin{proof} If $U, V$ are two distinct sets among the chosen ones, we want that precisely one of $U <V$, $V < U$ is true. \\
Suppose neither of these holds. Hence $\sigma(U) = \sigma(V)$. But then their intersection is non-empty, hence one is contained in the other, so $U < V$ or vice versa by \textbf{O2}.\\
Now, assume that both hold. Therefore, $l(\sigma(U)) = l(\sigma(V))$, none is contained in the other (otherwise we have $U \subset V$ and $V\subset U$, thus $U = V$), they are of same type, and further $\sigma(U) = \sigma(V)$. But as before, these must intersect, leading us to a contradiction.\\
Having proved the trichotomy, we now move to establishing transitivity of the relation. Suppose we have three chosen sets $U, V, T$ such that $U < V< T$, from which $U \not= V$, $V \not= T$ follows. Further, we cannot have $T = U$ as this would imply $U < V < U$. Hence, we can assume that all three sets are distinct. If $\textbf{O1}$ holds for $U < V$ or $V < T$, then it also holds for $U, T$. So, assume this is not true. Then we have the following cases:
\begin{description}
\item Case 1: $U < V$ by \textbf{O2}, $V < T$ by \textbf{O2}. \\
Then $U \supsetneq V \supsetneq T$, so $U < T$ by \textbf{O2} as well.
\item Case 2: $U < V$ by \textbf{O2}, $V < T$ by \textbf{O3}. \\
By Proposition~\ref{typeProp}, we see that (as $V$ cannot be of $W$-type), $U$ and $V$ are of same type. As $V \subset U$ and $V \not \subset T$, then $U \not\subset T$, so either $T \subset U$ and $U < T$ by \textbf{O2}, or $U < T$ by \textbf{O3}.
\item Case 3: $U < V$ by \textbf{O2}, $V < T$ by \textbf{O4}. \\
As above, we get $U \not\subset T$, so either $T \subset U$ implying $T > U$ or if $U$ is $B$-type and $V$ is $W$-type, so $U < T$ by \textbf{O3}, or by Proposition~\ref{typeProp} $U$ and $V$ have the same type, hence so does $T$, and $\sigma(U) = \sigma(V)$ (because of inclusion) so $U < T$ by \textbf{O4}.
\item Case 4: $U < V$ by \textbf{O3}, $V < T$ by \textbf{O2}. \\
If $T\supset U$, then $U \subset V$, which is impossible, so $T \not \supset U$. If $T \subset U$, we're done, otherwise from Proposition~\ref{typeProp} we obtain $U < T$ by \textbf{O3}.
\item Case 5: $U < V$ by \textbf{O3}, $V < T$ by \textbf{O3}. \\
Thus $U$ is of $A$-type, $V$ is of $B$-type, $T$ is of $W$-type. By Proposition~\ref{typeProp}, we conclude that $U \not \subset T$, $T \not \subset U$, so $U < T$ by \textbf{O3}.
\item Case 6: $U < V$ by \textbf{O3}, $V < T$ by \textbf{O4}. \\
By Proposition~\ref{typeProp}, $U \not\subset T$, so either $T \subset U$ giving $U < T$ by \textbf{O2}, or none is contained in the other and $U < T$ by \textbf{O3}.
\item Case 7: $U < V$ by \textbf{O4}, $V < T$ by \textbf{O2}. \\
Having $U$ as a subset of $T$ implies $U \subset V$ which is impossible. So $U \not\subset T$ and if $T \subset U$, then $U < T$ by \textbf{O2}. Otherwise, unless $U$ and $T$ are of same type, we have $U < T$ by \textbf{O3}, due to Proposition~\ref{typeProp}. Finally, from inclusion we deduce $\sigma(T) = \sigma(U)$ and thus $U < T$ by \textbf{O4}.
\item Case 8: $U < V$ by \textbf{O4}, $V < T$ by \textbf{O3}. \\
If $U \subset T$, by Proposition~\ref{typeProp}, we reach contradiction. Hence $U \supset T$, so $U < T$ by \textbf{O2}, or $U < T$ by \textbf{O3} otherwise.
\item Case 9: $U < V$ by \textbf{O4}, $V < T$ by \textbf{O4}. \\
If $T \supset U$ or vice versa, we have $\sigma(T) = \sigma(U)$ which is impossible. As all three sets are of same types, we get $U < T$ by \textbf{O4}.
\end{description}
Finally, given a subset $P$ of $\mathcal{D}$, consider its subset $P'$ of those sets $U \in P$ such that $l(\sigma(U)) = \min\{l(\sigma(V)):V\in P\}$. This is finite by the Proposition~\ref{finProp}, hence we can find the minimal element of $P'$ with respect to $<$, which is smaller than any member of $P \setminus P'$ by \textbf{O1}, making $\mathcal{D}$ well-ordered under $<$.
\end{proof}
Hence, as $\mathcal{D}$ is countable, we can take $S_0 = \min \mathcal{D}$ and for $k \geq 1$, $S_{k} = \min (\mathcal{D} \setminus \{S_0, S_1,\dots ,S_k\})$. Observe that for any given $U \in \mathcal{D}$ there are only finitely many words of length at most $l(\sigma(U))$. The Proposition~\ref{finProp} then tells us that there are only finitely many $V \in \mathcal{D}$ such that $l(\sigma(V)) \leq l(\sigma(U))$. Hence, there are only finitely many $V \in \mathcal{D}$ such that $V < U$, so $U = S_{k}$ for some $k \in \mathbb{N}$. Therefore, $\{S_0, S_1, S_2, \dots\} = \mathcal{D}$. Note also that $S_0 = X = W_\emptyset$ being the only set with shortest word whose length is non-positive.

\subsection{Proof that $\mathcal{D}$ satisfies the required properties}
All that is left is to show that $S_0, S_1, \dots$ satisfy \textbf{A1} - \textbf{A6}. Having done the most of the work already, the proof of the following claims will either be the rather short or the obvious case-examination.
\begin{proposition} $(S_n)_{n \geq 0}$ satisfy \textbf{A1}.\end{proposition}
\begin{proof} Let $S_i, S_j$ be such that $i < j$ and they intersect. Thus, $S_i < S_j$, $l(\sigma(S_i)) \leq l(\sigma(S_j))$ and also $S_i \subset S_j$ or $S_i \supset S_j$, by Proposition~\ref{containProp}. But if $S_i \subset S_j$, then we must have $l(\sigma(S_i)) = l(\sigma(S_j))$ and hence $S_j < S_i$ by \textbf{O2}, which is contradiction. Also $S_0 = X$.\end{proof}
\begin{proposition} $(S_n)_{n \geq 0}$ satisfy \textbf{A2}.\end{proposition}
Note that this is the Proposition~\ref{finProp}, but we include this as another claim for completeness.
\begin{proposition} $(S_n)_{n \geq 0}$ satisfy \textbf{A3}.\end{proposition}
\begin{proof} Let $N$ be given. By the Corollary~\ref{structCor}, given $S_i$ we can find $S_j, S_k$, disjoint subsets of $S_i$, which cover all but finitely many elements of $S_i$, therefore $i < j,k$ by \textbf{O1} or \textbf{O2}. So, we can start with $S_0$, and perform such splits until we are left with sets $S_{m_1}, S_{m_2}, \dots, S_{m_k}$, with $m_1, m_2, \dots , m_k > N$, which cover almost all elements of $X$, as in each split we lose only finitely many elements. \end{proof}

\begin{proposition} $(S_n)_{n \geq 0}$ satisfy \textbf{A4}.\end{proposition}
\begin{proof} Suppose that $S_{i_1} \supset S_{i_2} \supset \dots$ for some $i_1 < i_2 < \dots$. As usual, we consider different cases.
\begin{description}
\item Case 1: There is set of $A$-type among these.\\
Let $S_{i_k}$ be such a set. As $A$-type set splits into $A$-type sets, we have that when $n \geq k$, $S_{i_n} = A_{w, p_n, r_n}$, some minimal word $w$, and integers $p_n, r_n$, and $p_{n+1} > p_n$, so in fact $p_{n+1} \geq 2p_n$, as these are powers of 2. In particular, we deduce that for sufficiently large $n$, $p_n > n l(w)$ and thus $w^n$ is prefix of all words in $S_{i_n}$ implying $S_{i_n} \subset W_{w^n}$. By the Proposition~\ref{addavailProp}, one of $aw^n, bw^n$ is available, w.l.o.g. the former is true, hence $aS_{i_n} \subset W_{aw^n}= S_{j_n}$, some $j_n$.\\
Consider cyclic permutation $u$ of $w$, such that $sw = us$, $s$ being the last character of $w$. Let $v$ be a word such that $v^\infty = u^\infty$ and $l(u) \geq l(v)$. But then for cyclic permutation $t$ of $v$ such that for character $s'$ we have $s't = vs'$, we have $s' t^\infty = v^\infty = s w^\infty$, so $l(u) = l(t) \geq l(w) = l(v)$, proving the minimality of $v$. If $s = a$, we would have that $aw^n$ is not available as it would be initial segment of $u^\infty$, but $l(u) = l(w)$. Hence $s = b$, and $bS_{i_n} \subset A_{v, p_n, r} = S_{j_n}$, where $r = r_n+1$, unless $r_n = p_n - 1$ and then $r = 0$, $j_n$ suitable index. As $j_n$ tends to infinity as $n$ does, $\textbf{A4}$ holds in this case.
\item Case 2: There are no $A$-type sets, but there infinitely many of $W$-type.\\
Denote $W$-type sets among these by $W_{w_1} \supsetneq W_{w_2} \supsetneq \dots$. As $w_{i}$ is prefix of $w_{i+1}$, for all $i$, these define an infinite word $w$ whose initial segments the words $w_i$ are. By the Proposition~\ref{addavailProp}, w.l.o.g. $aw_i$ is available infinitely often, so we can take $W_{aw_{i_1}}$ for suitable $i_1 < i_2 < \dots$, to establish $a$ part of the claim.\\
Similarly, if there are infinitely many initial segments of $bw$ available, choosing these and their corresponding $W$-type sets establishes the claim. Suppose contrary, i. e. there is $m$ such that prefixes of $bw$ of length greater or equal to $m$ are all forbidden. Denote by $u_n$ initial segment of $bw$ of length $n$, and let $t_n$ be the shortest word such that $l(t_n) ^ 2 < l(u_n)$ and $u_n$ is initial segment of $t_n^\infty$. Now, suppose there is no $n \geq m$ such that $l(t_{n+1}) > l(t_n)$, hence $l(t_m) \geq l(t_{m+1}) \geq \dots$, so $bw = t_n^\infty$, some $n$. But these means that $w_i$ are forbidden from some point, resulting in contradiction. So, there must be such an $n \geq m$, so pick the smallest possible. Hence $u_{n+1}$ is not the initial segment of $t_n^\infty$, but $u_n$ is both initial segment of $t_n^\infty$ and $t_{n+1}^\infty$. Further, $l(u_n) \geq l(t_n)^2, l(t_{n+1})^2$ so as before, $t_n^\infty = t_{n+1}^\infty$, but $u_{n+1}$ is prefix of the later, but not former, giving us contradiction.
\item Case 3: Almost all sets are of $B$-type.
Given $S_{i_n} = B_{w, k}$, let $s$ be the last character of $w$, so $sw^{l(w)}$ is forbidden. Let $u$ be the cyclic permutation of $w$ such that $sw = us$, and so $u$ is minimal as well, by arguments in Case 1 of this proof. Then, $sB_{w,k} = B_{u, l}$ or $sB_{w,k} = B_{u, l} \setminus W_{v}$ for suitable index $l$ and word $v$, so take $S_{j_n} = B_{u, l}$. Let $s'$ be the character not equal to $s$ and $u$ be $\sigma(B_{w,k})$ without the last character. Then as $s u$ is forbidden, so $s'u$ is available and $s'B_{w,k} \subset W_{s'u} = S_{j_n}$. As $j_n$'s tend to infinity as $n$ does, we're done.
\end{description}
\end{proof}
\begin{proposition} $(S_n)_{n \geq 0}$ satisfy \textbf{A5}.\end{proposition}
\begin{proof} Let $S_n$ be given. If it is of $A$ or $W$-type, $S_n \subset W_{\sigma(S_n)}$, so choose character $s$ such that $s\sigma(S_n)$ is available and hence $sS_n \subset W_{s\sigma(S_n)}$ and $S_n < W_{s\sigma(S_n)}$, since $l(\sigma(S_n)) < l(\sigma(W_{s\sigma(S_n)}))$. If $S_n$ is of $B$-type, then let $u$ be the shortest word in $S_n$ after erasing the last character, so $S_n \subset W_u$ and once again choose $s \in \{a,b\}$ for which $su$ is available, hence $sS_n \subset W_{su}$. If $W_{su}$ intersects $W_u$, it must be its subset and $u$ prefix of $su$, but then $u = s^{l(u)}$, so $su$ is forbidden, as $l(u) \geq 1$. So $W_{su}$ and $W_u$ are disjoint, hence $W_{su}$ and $S_n$ are also. Combining this with $l(\sigma(S_n)) = l(\sigma(W_{su})$ and comparing the types gives $S_n < W_{su}$ by \textbf{O3}.\end{proof}

\begin{proposition} $(S_n)_{n \geq 0}$ satisfy \textbf{A6}.\end{proposition}
\begin{proof} Let $w \not=\emptyset$ be given. We will divide the proof into three sections, each showing the claim for particular set type.
\begin{itemize}
\item Suppose for some word $u$ we have $u, wu \in S_n = A_{t, p, r}$, where $t$ is minimal. Hence $l(u) > l(t)^2$. Suppose further $l(w) < l(t)$. So $w$ is initial segment of $u$, as $u$ is initial segment of $wu$. But then, $w^2$ is also initial segment of $u$, etc. up to $w^{l(t)}$ as $l(u) > l(t)^2 > l(w)l(t)$. Hence $t^{l(w)} = w ^{l(t)}$, but $t$ is minimal so $l(w) \geq l(t)$. Hence we have only finitely many choices for $A$-type $S_n$, as $l(w) \geq l(t)$ and $p \leq l(w)$, must hold. 
\item If $S_n$ is of $W$-type, say $W_v$, we have $u,wu \in W_v$, $v$ available. As there are only finitely many $v$ such that $l(v) \leq l(w)$, w.l.o.g. $l(v) > l(w)$. We have $u = vr_1$, some word $r_1$. Hence $wvr_1 \in W_v$ so $wv \in W_v$. But then $wv = vr_2$, some $r_2$, so $v = wv_1$, for some $v_1$, as $l(w)  < l(v)$, and $wwv_1 = wv_1r_2$ implies $wv_1 = v_1r_2$. We can iterate this until $v = w^kv_k, l(v_k) < l(w)$. But $wv_k = v_kr_{k+1}$, some $r_{k+1}$, making $v$ initial segment of $w^\infty$, but it is available hence $l(v) \leq l(w)^2$, therefore we have finitely many choices for $S_n$ of $W$-type.
\item Finally, suppose $S_n$ is $B$-type, and $u, wu \in S_n$. Then we have two available words $w_1s_1, w_2s_2$, where $s_1, s_2$ are characters, $w_1, w_2$ prefixes of $t^\infty$, some minimal word $t$, $u\in W_{w_1s_1}, wu \in W_{w_2s_2}, l(w_1), l(w_2) \geq l(t)^2$. So $u = w_1s_1r_1$, $wu = w_2s_2r_2$ holds for some words $r_1, r_2$.\\
    Suppose $l(w) \leq l(t)$. Let $wu'$ be the initial segment of $wu$ of length $l(t)^2$ thus it is prefix of $w_2$ hence of $t^\infty$. Thus $wt^{l(t)-1}$ is prefix of $t^{l(t)}$. Hence, $w^k t^{l(t) - k}$ is prefix of $t^{l(t)}$ for all $k \in [l(t)]$. In particular, $w^{l(t)}$ is prefix of $t ^ {l(t)}$, thus $w^{l(t)} = t^{l(w)}$, implying $w^\infty = t^\infty$, but $t$ is minimal, so $w = t^\alpha$, some $\alpha \geq 1$, so $w=t$ as $l(w) \leq l(t)$. Now we have $wu = w_2s_2r_2 = ww_1s_1r_1 = tw_1s_1r_1$, which differ from $t^\infty$ for the first time at $s_1$ and $s_2$, hence $w_2 = tw_1$. But for only finitely many $k$  can we have $B_{w,k}$ having such $W_{w_1s_1}, W_{w_2s_2}$ as subsets due to construction of $B$-type sets, and the fact that the common difference of arithmetic progression $I_k$ is at least $k/4$.\\
    Now, suppose $l(w) > l(t)$. If the claim is to be false, we can assume without loss of generality, that for infinitely many $k$ we have some $u \in B_{t, k}$ and $wu \in B_{t,k}$ too. Hence we can assume that there are $w_1, w_2$ prefixes of $t^\infty$, with some characters $s_1, s_2$, such that $w_1s_1, w_2s_2$ are available, $u \in W_{w_1s_1}, wu \in W_{w_2s_2}$, $w_1, w_2$ arbitrarily long, thus say $l(w_1), l(w_2) > l(w)^2$. So $u = w_1s_1r_1$, $wu = w_2s_2r_2$ for some words $r_1, r_2$, thus $wu = ww_1s_1r_1 = w_2s_2r_2$. Hence $t^{l(w)}$ is prefix of $w_2$ as $l(w_1) > l(w)^2 > l(t)^2$, and similarly $wt^{l(w)-1}$ is initial segment of $ww_1$ so both are prefixes of $wu$ hence $t^{l(w)}$ is prefix of $wt^{l(w)-1}$, which is then prefix of $w^2 t^{l(2)-2}$, etc. and prefix of $w^{l(w)}$, so $w^\infty = t^\infty$ and due to minimality $w = t^k$, some $k \geq 2$. Hence $w_2s_2r_2 = wu = t^kw_1s_1r_1$ so the character where $wu$ first differs from $t^\infty$ is at the same time at $s_1$ and $s_2$, hence $w_2 = t^kw_1 = ww_1$, implying $l(w_2) = l(w_1) + l(w)$, but as previously explained, this can occur just for finitely many $B_{t,k}$, which proves the claim.
\end{itemize}    
\end{proof}
Having showed the required statements we are ready to conclude:\\

\textbf{Theorem.} Given $\lambda \in (0,1)$, there is a compact (pseudo-)metric space $(X, d)$ on which we have continuous functions $f, g: X \rightarrow X$ such that given $x, y \in X$ either $d(f(x), f(y)) \leq \lambda d(x,y)$ or $d(g(x), g(y)) \leq \lambda d(x,y)$ holds, but no word in $f, g$ has a fixed point.\\
\section{Further remarks}
Even though we have given an answer to Austin's question and Stein's conjecture, some natural questions remain to be answered. Indeed, what kind of additional assumptions are necessary to ensure that Stein's conjecture is true. Here we mention two theorems related to our problem that could indicate what approach should be taken.\\
\begin{theorem}[\cite{Merry2}]\label{cthmMerryStein} Let $f$ be an operator on a complete metric space such that $\{f, f^2, \dots, f^n\}$ is a contractive family for some $n$. Then $f$ has a unique fixed point.\end{theorem}
\begin{theorem}[\cite{Austin}]\label{cthmAustin} Let $f,g$ be two operators on a complete metric space such that $\{f,g\}$ is a contractive family and $f,g$ commute. Then, $f$ and $g$ have unique common fixed point.\end{theorem}
The theorem~\ref{cthmMerryStein} was proved by Merryfield and Stein in~\cite{Merry2}, and later Merryfield, Rothschild and Stein removed the assumption of continuity in~\cite{Merry1}; this statement is also known as the Generalized Banach Contraction Theorem. In~\cite{Austin}, Austin proved the Theorem~\ref{cthmAustin} and conjectured that the result holds for arbitrarily many commuting operators.\\
Theorem~\ref{thm1} tells us that one should not look for a topological condition, and that algebraic condition as in the theorem of Austin or in the Generalized Banach Contraction Theorem is more convenient. The later type of the condition directly defines the geometry of the problem, and in general the assumption should, vaguely speaking, provide us with a finite-dimensional problem, and forbid compositions of functions behaving very badly, i.e. prevent the tree structure used in the counterexample from occurring. An appropriate assumption could be that $|\{f_{i_1} \circ f_{i_2} \circ \dots \circ f_{i_m}:i_1, \dots, i_m \in \{1,2,\dots, n\}\}|$ is of subexponential or polynomial growth in $m$.\\
To conclude, we raise the following question and conjecture which are extensions of a conjecture of Austin in~\cite{Austin}.
\begin{question} What algebraic conditions should be imposed on a finite contractive family $\{f_1, f_2, \dots, f_n\}$ of operators on a complete metric space to ensure the existence of a composition of the functions $f_i$ with a fixed point? \end{question}
\begin{conjecture} Let $\{f_1, f_2, \dots, f_n\}$ be a finite contractive family of operators on a complete metric space. Suppose that there are $C > 0, k \in \mathbb{N}$ such that $|\{f_{i_1} \circ f_{i_2} \circ \dots \circ f_{i_m}:i_1, \dots, i_m \in \{1,2,\dots, n\}\}| \leq Cm^k$. Then $f_1, \dots, f_n$ have a unique common fixed point.\end{conjecture}
\textbf{Acknowledgements.} My thanks go to Trinity College whose funding and support made this project possible.\\ 

\begin{flushleft}
Luka Mili\'{c}evi\'{c}\\
Trinity College,\\
Cambridge CB2 1TQ\\
United Kingdom\\
lm497@cam.ac.uk
\end{flushleft}
\end{document}